\documentclass[a4paper]{amsart}
\usepackage{wasowicz_hyperref}
\usepackage{times}
\usepackage{graphicx}

\info{Manuscript}

\author{Andrzej Komisarski}
\address[Andrzej Komisarski]{Faculty of Mathematics and Computer Science, University of Lodz, Banacha 22, 90--238 Łódź, Poland}
\email{\color{blue}andkom@math.uni.lodz.pl}

\author{\SW}
\address[\SW]{\SWaddr}
\email{\SWmail}

\title{Inequalities between remainders of quadratures}
\date{\today}

\theoremstyle{plain}
 \newtheorem*{Peano}{Peano Kernel Theorem}

\begin{document}
\begin{abstract}
It is well-known that in the class of convex functions the (nonnegative) remainder of the Midpoint Rule of the approximate integration is majorized by the remainder of the Trapezoid Rule. Hence the approximation of the integral of the convex function by the Midpoint Rule is better than the analogous approximation by the Trapezoid Rule. Following this fact we examine remainders of certain quadratures in the classes of convex functions of higher orders. Our main results state that for $3$-convex ($5$-convex, respectively) functions the remainder of the  $2$-point ($3$-point, respectively) Gauss quadrature is non-negative and it is not greater than the remainder of the Simpson's Rule ($4$-point Lobatto quadrature, respectively). We also check the $2$-point Radau quadratures for $2$-convex functions to demonstrate that similar results fail to hold for convex functions of even orders. We apply Peano Kernel Theorem as a~main tool of our considerations.
\end{abstract}
\maketitle

\section{Introduction}

Our starting point is the celebrated Hermite--Hadamard inequality stating that if a~function $f\colon\I\to\R$ defined on a~real interval $\I$ is convex, then
\begin{equation}\label{eq:HH}
 f\biggl(\frac{x+y}{2}\biggr)\xle \frac{1}{y-x}\int_x^yf(t)\,\dd t\xle\frac{f(x)+f(y)}{2}
\end{equation}
for any $x,y\in\I$, $x\ne y$. If we apply the second inequality of~\eqref{eq:HH} to the intervals with endpoints $x,\frac{x+y}{2}$ and $\frac{x+y}{2},y$, sum these inequalities side by side, we get
\begin{equation}\label{eq:HH2}
\frac{2}{y-x}\int_x^yf(t)\,\dd t\xle f\biggl(\frac{x+y}{2}\biggr)+\frac{f(x)+f(y)}{2}\,.
\end{equation}
After a~slight rearrangement, we arrive at
\begin{equation}\label{eq:HH1}
 0\xle \frac{1}{y-x}\int_x^yf(t)\,\dd t-f\biggl(\frac{x+y}{2}\biggr)\xle \frac{f(x)+f(y)}{2}-\frac{1}{y-x}\int_x^yf(t)\,\dd t\,.
\end{equation}
Of course both terms above are non-negative by the first inequality of~\eqref{eq:HH}. Observe that~\eqref{eq:HH1} states that for a~convex function $f\colon\I\to\R$ and for any $x,y\in\I$ the remainder of the Midpoint Rule of the approximate integration is non-negative and not greater than the remainder of the Trapezoid Rule. This fact is well-known (\cf~\eg\  \cite[p.~52, Remark 1.9.3]{NicPer06}). We recalled the proof because of the methods used in this paper.

In 1926 Hopf~\cite{Hop26} considered the functions with non-negative divided differences:
\[
 [x_1;f]=f(x_1)\quad\text{and}\quad
 [x_1,\dots,x_{n+1};f]=\frac{[x_2,\dots,x_{n+1};f]-[x_1,\dots,x_n;f]}{x_{n+1}-x_1}\,.
\]
The function $f\colon\I\to\R$ with $[x_1,\dots,x_{n+2};f]\xge 0$ (for any $x_1,\dots,x_{n+2}\in\I$) is called (after Popoviciu~\cite{Pop34}) $n$-\emph{convex}. Let us remark that the other authors (like Bullen~\cite{Bul71} or Pinkus and Wulbert~\cite{PinWul05}) use in this context the terminology $(n+1)$-convex function. Both conventions have both advantages and disadvantages. For some discussion of this matter see~\cite{Was10MIA}. We will stay with Popoviciu's definition (here $1$-convexity means ordinary convexity). Let us mention that $n$-convex functions have nice regularity properties. In particular, they are $(n-1)$-times continuously differentiable on the interior of a~domain with  convex derivatives of order $n-1$; they are also Riemann integrable on each interval $[a,b]\subset\I$. Furthermore, if $f$ is $(n+1)$-times differentiable on $\I$, then $f$ is $n$-convex if and only if $f^{(n+1)}\xge 0$ on~$\I$. For these properties see for instance \cite{Kuc09,PinWul05,Pop34,RobVar73}. Nevertheless, not every $n$-convex function is so regular. The example is the spline function $f(x)=x^n_+$, where $x_+=\frac{x+|x|}{2}$, for which $f^{(n)}(0)$ does not exist. But for the purposes of the approximate integration even the assumption upon the integrand to be $(n+1)$-times continuously differentiable seems to be not too restrictive.

Bessenyei and P\'ales obtained in \cite{BesPal02} many inequalities of Hermite--Hadamard type for convex functions of higher orders, among them 
  \begin{equation}\label{eq:Radau}
   \frac{1}{4}f(x)+\frac{3}{4}f\biggl(\frac{x+2y}{3}\biggr)\xle\frac{1}{y-x}\int_x^y f(t)\,\dd t\xle\frac{3}{4}f\biggl(\frac{2x+y}{3}\biggr)+\frac{1}{4}f(y)
  \end{equation}
  for any $x,y\in\I$ \st~$x<y$, provided $f\colon\I\to\R$ is $2$-convex, 
  \begin{multline}\label{eq:Gauss2}
   \frac{1}{2}f\biggl(\frac{3+\sqrt{3}}{6}x+\frac{3-\sqrt{3}}{6}y\biggr)+\frac{1}{2}f\biggl(\frac{3-\sqrt{3}}{6}x+\frac{3+\sqrt{3}}{6}y\biggr)\\\xle\frac{1}{y-x}\int_x^y f(t)\,\dd t\xle\frac{1}{6}f(x)+\frac{2}{3}f\biggl(\frac{x+y}{2}\biggr)+\frac{1}{6}f(y)
  \end{multline}
  for any $x,y\in\I$ provided $f\colon\I\to\R$ is $3$-convex and
  \begin{multline}\label{eq:Gauss3}
   \frac{5}{18}f\biggl(\frac{5+\sqrt{15}}{10}x+\frac{5-\sqrt{15}}{10}y\biggr)+\frac{4}{9}f\biggl(\frac{x+y}{2}\biggr)\\+\frac{5}{18}f\biggl(\frac{5-\sqrt{15}}{10}x+\frac{5+\sqrt{15}}{10}y\biggr)
   \xle\frac{1}{y-x}\int_x^y f(t)\,\dd t\\
   \xle \frac{1}{12}f(x)+\frac{5}{12}f\biggl(\frac{5+\sqrt{5}}{10}x+\frac{5-\sqrt{5}}{10}y\biggr)\\+\frac{5}{12}f\biggl(\frac{5-\sqrt{5}}{10}x+\frac{5+\sqrt{5}}{10}y\biggr)+\frac{1}{12}f(y)
  \end{multline}
  for any $x,y\in\I$ provided $f\colon\I\to\R$ is $5$-convex.
  They used the method of smoothing. The author reproved these inequalities by a~support technique (\cf~\cite{Was07JMAA,Was09JMAA}). If $[-1,1]\subset\I$ and $x=-1,y=1$, then inequalities~\eqref{eq:Radau}, \eqref{eq:Gauss2}, \eqref{eq:Gauss3} reduce to
  \begin{align}
   \label{eq:Radau_q}\Rad_2^{\ell}[f]&\xle\int_{-1}^1 f(x)\,\dd x\xle\Rad_2^r[f],\\[1ex]
   \label{eq:Gauss2_q}\G_2[f]&\xle\int_{-1}^1 f(x)\,\dd x\xle\Lob_3[f],\\[1ex]
   \label{eq:Gauss3_q}\G_3[f]&\xle\int_{-1}^1 f(x)\,\dd x\xle\Lob_4[f],
  \end{align}
respectively. The functionals above are very well known in the numerical analysis Radau, Gauss-Legendre and Lobatto quadratures:
  \begin{align*}
   \Rad_2^{\ell}[f]&=\tfrac{1}{2}f(-1)+\tfrac{3}{2}f\bigl(\tfrac{1}{3}\bigr)&&\text{($2$-point left Radau),}\\[1ex]
   \Rad_2^r[f]&=\tfrac{3}{2}f\bigl(-\tfrac{1}{3}\bigr)+\tfrac{1}{2}f(1)&&\text{($2$-point right Radau),}\\[1ex]
   \G_2[f]&=f\bigl(-\tfrac{\sqrt{3}}{3}\bigr)+f\bigl(\tfrac{\sqrt{3}}{3}\bigr)&&\text{($2$-point Gauss-Legendre),}\\[1ex]
   \G_3[f]&=\tfrac{5}{9}f\bigl(-\tfrac{\sqrt{15}}{5}\bigr)+\tfrac{8}{9}f(0)+\tfrac{5}{9}f\bigl(\tfrac{\sqrt{15}}{5}\bigr)&&\text{($3$-point Gauss-Legendre),}\\[1ex]
   \Lob_3[f]&=\tfrac{1}{3}f(-1)+\tfrac{4}{3}f(0)+\tfrac{1}{3}f(1)&&\text{($3$-point Lobatto, Simpson's Rule),}\\[1ex]
	\Lob_4[f]&=\tfrac{1}{6}f(-1)+\tfrac{5}{6}f\bigl(-\tfrac{\sqrt{5}}{5}\bigr)+\tfrac{5}{6}f\bigl(\tfrac{\sqrt{5}}{5}\bigr)+\tfrac{1}{6}f(1)&&\text{($4$-point Lobatto).}
  \end{align*}
  
In this paper, following~\eqref{eq:HH1}, we examine the remainders of quadratures $\G_n$ and $\Lob_{n+1}$ (for $n\in\{2,3\}$) to investigate which one gives better approximation of the integral of the $(2n-1)$-convex function. Such a~comparison is justified because both quadratures are exact on polynomials on order $2n-1$. We prove that
\begin{equation}\label{eq:remainders}
 0\xle\int_{-1}^1 f(t)\,\dd t-\G_n[f]\xle\Lob_{n+1}[f]-\int_{-1}^1f(t)\,\dd t,
\end{equation}
whenever $f\in\C^{2n}[-1,1]$ is $(2n-1)$-convex. That both above remainders are nonnegative follows by~\eqref{eq:Gauss2_q}, \eqref{eq:Gauss3_q}. The second inequality of~\eqref{eq:remainders} is obviously equivalent to
\[
 \frac{1}{2}\G_n[f]+\frac{1}{2}\Lob_{n+1}[f]-\int_{-1}^1 f(t)\,\dd t\xge 0.
\]
For that reason, to obtain~\eqref{eq:remainders}, we need to determine the error terms of the quadratures $\frac{1}{2}\G_n+\frac{1}{2}\Lob_{n+1}$ for $n\in\{2,3\}$.

The analogous inequality is not valid for Radau quadratures $\Rad_2^{\ell},\Rad_2^r$ (in the class of $2$-convex functions). This tells us that it is impossible to say which Radau $2$-point quadrature gives better approximation of the integral of the $2$-convex function.

\begin{exmp}
 The functions $f(x)=\frac{x+1}{x+2}$, $g(x)=(x+2)^4$ are $2$-convex on $[-1,1]$ (because $f'''\xge 0$, $g'''\xge 0$ on $[-1,1]$). We have
 \begin{align*}
  &\frac{1}{2}\Rad_2^{\ell}[f]+\frac{1}{2}\Rad_2^r[f]-\int_{-1}^1 f(t)\,\dd t=3-\ln\frac{116}{105}>0,\\[1ex]
  &\frac{1}{2}\Rad_2^{\ell}[g]+\frac{1}{2}\Rad_2^r[g]-\int_{-1}^1 g(t)\,\dd t=\frac{16}{135}<0.
 \end{align*}
\end{exmp}

\section{Error terms of certain quadratures}

In this section, we determine the error terms of the quadratures $\frac{1}{2}\G_n+\frac{1}{2}\Lob_{n+1}$ for $n\in\{2,3\}$. To this end we use Peano Kernel Theorem (\cf~\cite{BraPet}). Let us recall it.
\begin{Peano}
 Let $E$ be a~continuous linear functional on $\C[a,b]$ such that $E[p]=0$ for any polynomial $p$ of degree at most $r-1$ for some $r\in\{1,2,\dots\}$. Let
 \[
  K(x)=E\biggl[\frac{(\cdot-x)_+^{r-1}}{(r-1)!}\biggr]\,,\quad x\in[a,b]
 \]
 be the Peano kernel of $E$. If $f\in\C^r[a,b]$ then
 \[
  E[f]=\int_a^b f^{(r)}(u)K(u)\,\dd u\,.
 \]
 If, in particular $K$ has no sign changes in $[a,b]$ then
 \[
  E[f]=f^{(r)}(\xi)\int_a^b K(u)\,\dd u
 \]
 for some $\xi\in[a,b]$.
\end{Peano}

Next we deal with the Peano kernels of the remainders of symmetric quadratures. Recall that the quadrature~$Q$ (defined for functions $f\colon[-1,1]\to\R$) is called \emph{symmetric}, if
\[
 Q[f]=w_0f(0)+\sum_{i=1}^kw_i\bigl(f(x_i)+f(-x_i)\bigr).
\]
For example, all Gauss-Legendre and Lobatto quadratures are symmetric, also the quadrature
\[
 \frac{1}{2}\G_n+\frac{1}{2}\Lob_{n+1}
\]
is symmetric.

\begin{lem}\label{lem_even}
 Let $Q$ be a~symmetric quadrature, which is exact for polynomials of order~$r-1$, where~$r\in\N$ is an even number (\ie{} $Q[p]=\int_{-1}^1 p(t)\,\dd t$ for all polynomials $p$~of order~$r-1$). Then the Peano kernel~$K$ of the remainder
 \[
  Q[f]-\int_{-1}^1 f(t)\,\dd t
 \]
 is an even function on $[-1,1]$.
\end{lem}

\begin{proof}
 The notation $Q[f]=Q\bigl[f(t)\bigr]$ is useful in this proof. For $x,t\in[-1,1]$ define $f_x(t)=(x-t)_+^{r-1}$, $g_x(t)=(t-x)_+^{r-1}$ and
 \begin{equation}\label{lem2:1}
  k(x)=Q\bigl[g_x(t)\bigr]-\int_{-1}^1g_x(t)\,\dd t.
 \end{equation}
 Since $K(x)=\frac{k(x)}{(r-1)!}$ is the Peano kernel of $E$, it is enough to check that~$k$ is~an even function.
 
 Observe that
 \[
  g_{-x}(t)=(t+x)_+^{r-1}=\bigl(x-(-t)\bigr)_+^{r-1}=f_x(-t).
 \]
 Because $Q$ is symmetric $Q\bigl[f(t)\bigr]=Q\bigl[f(-t)\bigr]$. Also the integral is symmetric. This leads to
 \begin{multline}\label{lem2:2}
  k(-x)=Q\bigl[g_{-x}(t)\bigr]-\int_{_1}^1 g_{-x}(t)\,\dd t
  =Q\bigl[f_x(-t)\bigr]-\int_{_1}^1 f_x(-t)\,\dd t\\
  =Q\bigl[f_x(t)\bigr]-\int_{_1}^1 f_x(t)\,\dd t.
 \end{multline}
 Since $r$ is even, $(x-t)^{r-1}=f_x(t)-g_x(t)$. This is a~polynomial (of a~variable~$t$) of degree~$r-1$. Then~$Q$ is precise on $f_x-g_x$, which means that
 \begin{equation}\label{lem2:3}
  Q\bigl[f_x(t)-g_x(t)\bigr]=\int_{-1}^1\bigl(f_x(t)-g_x(t)\bigr)\dd t.
 \end{equation}
 Using the equations~\eqref{lem2:1}, \eqref{lem2:2}, \eqref{lem2:3} we arrive at
 \[
  k(-x)=\biggl[Q\bigl[f_x(t)-g_x(t)\bigr]-\int_{-1}^1\bigl(f_x(t)-g_x(t)\bigr)\dd t\biggr]
  +\biggl[Q\bigl[g_x(t)\bigr]-\int_{-1}^1 g_x(t)\,\dd t\biggr]=k(x)
 \]
 and the proof is finished.
\end{proof}

Now we are in a position to present main results of this section.

\begin{prop}\label{prop:3-convex}
 If $f\in\C^4[-1,1]$ then
 \[
  \int_{-1}^1 f(t)\,\dd t=\frac{1}{2}\G_2[f]+\frac{1}{2}\Lob_3[f]-\frac{f^{(4)}(\xi)}{540}
 \]
 for some $\xi\in[-1,1]$.	
\end{prop}	
\begin{proof}
 Consider the linear (and continuous on $\C[-1,1]$) functional
 \[
  E[f]=\frac{1}{2}\G_2[f]+\frac{1}{2}\Lob_3[f]-\int_{-1}^1 f(t)\,\dd t\,. 
 \]
 Trivially $E[e_i]=0$ for the monomials $e_i(x)=x^i$, $i=0,1,2,3$, hence all the assumptions of the Peano Kernel Theorem are met for $r=4$. It is not difficult to determine the Peano Kernel of~$E$. This is, of course
 \[
 K(x)=E\biggl[\frac{(\cdot-x)_+^3}{6}\biggr]\,,\quad x\in[-1,1]\,.
 \]
 Taking into account that (for $x\in[-1,1]$), we have $(-1-x)_+=0$, $(1-x)_+=1-x$ and
 \[
  \int_{-1}^1 (t-x)_+^3\,\dd t=\int_x^1 (t-x)^3\dd t=\frac{1}{4}(1-x)^4
 \]
 we arrive at 
 \[
  K(x)=\frac{1}{36}\Bigl[3\bigl(-\tfrac{\sqrt{3}}{3}-x\bigr)_+^3+4(-x)_+^3+3\bigl(\tfrac{\sqrt{3}}{3}-x\bigr)_+^3\Bigr]+\frac{(1-x)^3(3x-1)}{72}\,.
 \]
 after a~bit of computation we obtain
 \[
  K(x)=
   \begin{cases}
    -\dfrac{1}{24}x^4+\dfrac{1}{18}x^3+\dfrac{\sqrt{3}-2}{12}x^2+\dfrac{2\sqrt{3}-3}{216}&\text{for }x\in\Bigl[0,\frac{\sqrt{3}}{3}\Bigr)\,,\\[2ex]
    \dfrac{(1-x)^3(3x-1)}{72}&\text{for }x\in\Bigl[\frac{\sqrt{3}}{3},1\Bigr]\,.
   \end{cases}
 \]
 It is easily seen that $K(x)\xge 0$ for $x\in\Bigl[\frac{\sqrt{3}}{3},1\Bigr]$. If $x\in\Bigl[0,\frac{\sqrt{3}}{3}\Bigr)$, then $K'(x)=-\frac{x(x^2-x-\sqrt{3}+2)}{6}\xle 0$ and $K(x)>0$ by $K\bigl(\frac{\sqrt{3}}{3}\bigr)=\frac{7\sqrt{3}-12}{162}>0$. Then $K(x)\xge 0$, $x\in[0,1]$. Since $K(-x)=K(x)$ (\cf{}~Lemma~\ref{lem_even}), $K\xge 0$ on $[-1,1]$. Peano Kernel Theorem implies that there exists $\xi\in[-1,1]$ \st
 \[
  E[f]=f^{(4)}(\xi)\int_{-1}^1 K(x)\,\dd x=\frac{f^{(4)}(\xi)}{540}\,,
 \]
 which completes the proof.
\end{proof}	

\begin{prop}\label{prop:5-convex}
	If $f\in\C^6[-1,1]$ then
	\[
	\int_{-1}^1 f(t)\,\dd t=\frac{1}{2}\G_3[f]+\frac{1}{2}\Lob_4[f]-\frac{f^{(6)}(\xi)}{94500}
	\]
	for some $\xi\in[-1,1]$.	
\end{prop}

\begin{proof}
 We start in the same way as in the proof of Proposition~\ref{prop:3-convex}. Consider the linear (and continuous on $\C[-1,1]$) functional
 \[
  E[f]=\frac{1}{2}\G_3[f]+\frac{1}{2}\Lob_4[f]-\int_{-1}^1 f(t)\,\dd t\,. 
 \]
 Trivially $E[e_i]=0$ for the monomials $e_i(x)=x^i$, $i=0,1,2,3,4,5$, hence all the assumptions of the Peano Kernel Theorem are met for $r=6$. The Peano Kernel of~$E$ has the form
 \[
 K(x)=E\biggl[\frac{(\cdot-x)_+^5}{120}\biggr]\,,\quad x\in[-1,1]\,.
 \]
 If $x\in[-1,1]$, then $(-1-x)_+=0$, $(1-x)_+=1-x$ and
 \[
  \int_{-1}^1 (t-x)_+^5\,\dd t=\int_x^1 (t-x)^5\dd t=\frac{1}{6}(1-x)^6.
 \]
 Therefore
 \begin{multline*}
  K(x)=\frac{1}{4320}
        \Bigl[
          10\bigl(-\tfrac{\sqrt{15}}{5}-x\bigr)_+^5
         +15\bigl(-\tfrac{\sqrt{5}}{5}-x\bigr)_+^5\\
         +16(-x)_+^5
         +15\bigl(\tfrac{\sqrt{5}}{5}-x\bigr)_+^5
         +10\bigl(\tfrac{\sqrt{15}}{5}-x\bigr)_+^5
         +3(1-x)^5(2x-1)
        \Bigr]\,.
 \end{multline*}
 To apply the Peano Kernel Theorem we need to prove that $K\xge 0$ on $[-1,1]$. By Lemma~\ref{lem_even} $K$ is an even function, so it is enough to check that $k(x)=4320K(x)\xge 0$ for all $x\in[0,1]$. The spline function $k$ has the form
 \[
  k(x)=15\bigl(\tfrac{\sqrt{5}}{5}-x\bigr)_+^5
       +10\bigl(\tfrac{\sqrt{15}}{5}-x\bigr)_+^5
       +3(1-x)^5(2x-1)\,,\quad x\in[0,1].
 \]
 If $\frac{\sqrt{15}}{5}<x\xle 1$, then trivially $k(x)=3(1-x)^5(2x-1)\xge 0$.
 \par\smallskip
 From now on our proof completely differs from the proof of Proposition~\ref{prop:3-convex}.
\par\smallskip
 Let us start with $x\in\Bigl(\frac{\sqrt{5}}{5},\frac{\sqrt{15}}{5}\Bigr]$. Then
 \[
  k(x)=10\bigl(\tfrac{\sqrt{15}}{5}-x\bigr)^5
       +3(1-x)^5(2x-1).
 \]
 If $\frac{1}{2}<x\xle\frac{\sqrt{15}}{5}$, then $\bigl(\frac{\sqrt{15}}{5}-x\bigr)^5\xge 0$ and $(1-x)^5(2x-1)>0$, whence $k(x)>0$. Assume now that $\frac{\sqrt{5}}{5}<x\xle\frac{1}{2}$. In this case
 \[
  0<\frac{1-x}{\frac{\sqrt{15}}{5}-x}=1+\frac{1-\frac{\sqrt{15}}{5}}{\frac{\sqrt{15}}{5}-x}\xle 1+\frac{1-\frac{\sqrt{15}}{5}}{\frac{\sqrt{15}}{5}-\frac{1}{2}}=\frac{1}{2\sqrt{\frac{3}{5}}-1}<\frac{1}{2\cdot\frac{29}{38}-1}=1.9,
 \]
 which gives us
 \[
  0<\Biggl(\frac{1-x}{\frac{\sqrt{15}}{5}-x}\Biggr)^5<1.9^5<30.
 \]
 On the other hand,
 \[
 0\xge 3(2x-1)\xge 3\biggl(2\cdot\frac{\sqrt{5}}{5}-1\biggr)=\frac{6}{\sqrt{5}}-3>\frac{6}{\frac{9}{4}}-3=-\frac{1}{3}.
 \]
 Therefore
 \[
  \frac{3(1-x)^5(2x-1)}{\bigl(\frac{\sqrt{15}}{5}-x\bigr)^5}+10>30\cdot\Bigl(-\frac{1}{3}\Bigr)+10=0.
 \]
 Finally
 \[
  k(x)=\biggl(\frac{\sqrt{15}}{5}-x\biggr)^5\Biggl[\frac{3(1-x)^5(2x-1)}{\bigl(\frac{\sqrt{15}}{5}-x\bigr)^5}+10\Biggr]>0
 \]
 on $\Bigl(\frac{\sqrt{5}}{5},\frac{1}{2}\Bigr]$.
 \par\smallskip
 It was left to check that $k(x)>0$ on $\Bigl[0,\frac{\sqrt{5}}{5}\Bigr]$. To this end we apply Budan-Fourier Theorem~(\cf~\cite{Con43} for a~formulation and the elementary proof). It states that the number of real roots belonging to the interval $(a,b]$ (counting multiplicities) of a~polynomial $p$ of degreee~$n$ with real coefficients equals to $V(a)-V(b)-2m$, where $m$ is either zero, or the natural number. The term $V(x)$ stands for the number of sign changes in the sequence
 \[
  \bigl(p(x),p'(x),p''(x),\dots,p^{(n)}(x)\bigr)
 \]
 at a~point $x$ with the convention that zeros are not counted.
 \par\smallskip
 For $0\xle x\xle\frac{\sqrt{5}}{5}$ the function $k$ has the form
 \begin{align*}
  k(x)&=15\bigl(\tfrac{\sqrt{5}}{5}-x\bigr)^5
       +10\bigl(\tfrac{\sqrt{15}}{5}-x\bigr)^5
       +3(1-x)^5(2x-1)\\
      &=-6x^6+8x^5+\bigl(15\sqrt{5}+10\sqrt{15}-75\bigr)x^4+\bigl(6\sqrt{5}+12\sqrt{15}-60\bigr)x^2\\[1ex]
      &\phantom{=\;}+\frac{3\sqrt{5}+18\sqrt{15}-75}{25}
 \end{align*}
 The signs of the terms of the sequence $\bigl(k(0),k'(0),\dots,k^{(6)}(0)\bigr)$ are simply the signs of the coefficients of~$k$ starting with a~constant term and finishing with the leading coefficient. The sequence of these signs is $(+,0,-,0,-,+,-)$, whence $V(0)=3$. Now our job is to determine~$V\bigl(\frac{\sqrt{5}}{5}\bigr)$. We would not like to bore the reader with the easy differentiation of~$k$ and long computations. Let us give only the final results, which could be checked either manually (as the authors did), or with the help of computer software (the authors, of course, did it too).
 \begin{align*}
  k\biggl(\frac{\sqrt{5}}{5}\biggr)&=\frac{8}{125}\bigl(31\sqrt{5}+55\sqrt{15}-282\bigr)>0\\[1ex]
  k'\biggl(\frac{\sqrt{5}}{5}\biggr)&=\frac{8}{125}\bigl(500\sqrt{3}-567\sqrt{5}+400\bigr)<0\\[1ex]
  k''\biggl(\frac{\sqrt{5}}{5}\biggr)&=54.4\sqrt{5}+48\sqrt{15}-307.2>0\\[1ex]
  k'''\biggl(\frac{\sqrt{5}}{5}\biggr)&=240\sqrt{3}-388.8\sqrt{5}+456>0\\[1ex]
  k^{(4)}\biggl(\frac{\sqrt{5}}{5}\biggr)&=552\sqrt{5}+240\sqrt{15}-2232<0\\[1ex]
  k^{(5)}\biggl(\frac{\sqrt{5}}{5}\biggr)&=-864\sqrt{5}+960<0\\[1ex]
  k^{(6)}\biggl(\frac{\sqrt{5}}{5}\biggr)&=-4320<0.
 \end{align*}
 Then we can see that $V\bigl(\frac{\sqrt{5}}{5}\bigr)=3$ and $V(0)-V\bigl(\frac{\sqrt{5}}{5}\bigr)=0$. Budan-Fourier Theorem implies now that~$k$ has no roots in $\Bigl(0,\frac{\sqrt{5}}{5}\Bigr]$. Because $k$ is positive at the endpoints of this interval, we infer that $k>0$ here.
 \par\smallskip
 We have shown that $k\xge 0$ on $[0,1]$. This means that the Peano kernel $K$ of the functional $E$ is nonnegative on $[-1,1]$. Peano Kernel Theorem implies that there exists $\xi\in[-1,1]$ \st
 \[
  E[f]=f^{(6)}(\xi)\int_{-1}^1 K(x)\,\dd x=\frac{f^{(6)}(\xi)}{94500}\,,
 \]
 which completes the proof.
 \end{proof}

\section{Main result}

Our main result is the immediate consequence of the above propositions.

\begin{thm}\label{th:remainder}
 Let $n\in\{2,3\}$ and $f\in\C^{2n}[-1,1]$ be $(2n-1)$-convex. Then
 \[
  0\xle\int_{-1}^1 f(t)\,\dd t-\G_n[f]\xle\Lob_{n+1}[f]-\int_{-1}^1f(t)\,\dd t.
 \] 
\end{thm}	

\begin{proof}
 The inequality
 \[
  0\xle\int_{-1}^1 f(t)\,\dd t-\G_n[f]
 \]
 follows by \eqref{eq:Gauss2_q}, \eqref{eq:Gauss3_q}. Because $f$ is $(2n-1)$-convex, we have $f^{(2n)}\xge 0$ on $[-1,1]$. Propositions~\ref{prop:3-convex},~\ref{prop:5-convex} yield
 \[
 \frac{1}{2}\G_n[f]+\frac{1}{2}\Lob_{n+1}[f]-\int_{-1}^1 f(t)\,\dd t\xge 0,
 \]
 which concludes the proof.
\end{proof}

Actually for $n>3$ the statement of Theorem~\ref{th:remainder} is an open problem. For $n\in\{4,5\}$ the author drew (with the Maxima CAS system) the graphs of the Peano kernels of the quadratures $\frac{1}{2}\G_n+\frac{1}{2}\Lob_{n+1}$. 

\begin{figure}[h]
\begin{center}
	\includegraphics[width=0.7\textwidth]{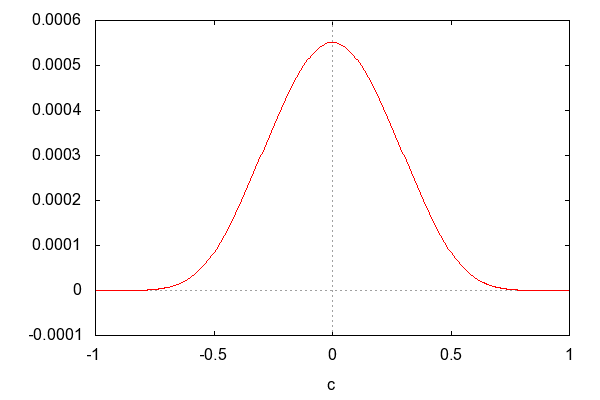}\\
\end{center}
\caption{Peano kernel $K$ of the remainder of the quadrature $\frac{1}{2}\G_4[f]+\frac{1}{2}\Lob_{5}[f]$}
\end{figure}

\begin{figure}[h]
\begin{center}
	\includegraphics[width=0.7\textwidth]{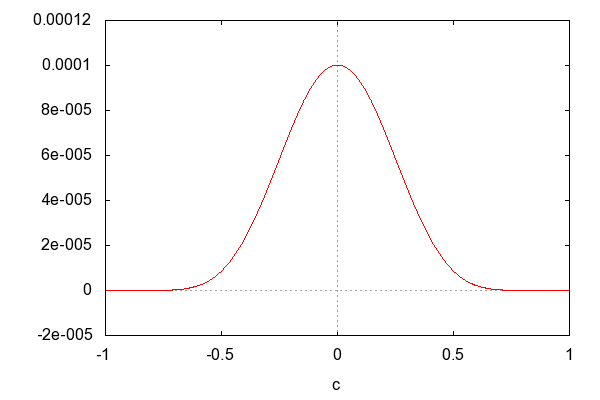}\\
\end{center}
\caption{Peano kernel $K$ of the remainder of the quadrature $\frac{1}{2}\G_5[f]+\frac{1}{2}\Lob_{6}[f]$}
\end{figure}

As we can see, both of them are nonnegative, so the authors are almost sure that Theorem~\ref{th:remainder} remains valid for $n\in\{4,5\}$ as well as for any $n\in\N$.

\end{document}